\documentclass{amsart}

\usepackage[T1]{fontenc}

\usepackage[margin=2cm]{geometry}

\usepackage{mathtools}

\usepackage{amssymb}

\usepackage{bm}

\usepackage{bbm}

\usepackage{esint}

\usepackage[thinc]{esdiff}

\usepackage[overload]{empheq}

\newcommand{\dd}{\mathop{}\!\mathrm{d}}

\usepackage{xcolor}

\usepackage{graphicx}

\graphicspath{{./Images/}}

\usepackage{subfig}

\usepackage{transparent}

\usepackage{eso-pic}

\usepackage{tikz}

\usepackage{float}

\usepackage{tabularx}

\usepackage{algorithm}

\usepackage{algorithmic}

\theoremstyle{plain}

\newtheorem{theorem}{Theorem}[section]
\newtheorem{lemma}[theorem]{Lemma}
\newtheorem{proposition}[theorem]{Proposition}
\newtheorem{corollary}[theorem]{Corollary}

\theoremstyle{definition}

\newtheorem{definition}[theorem]{Definition}
\newtheorem{example}[theorem]{Example}

\theoremstyle{remark}

\newtheorem{remark}[theorem]{Remark}

\usepackage[
    square,
    numbers,
    sort&compress
]{natbib}

\usepackage[
    colorlinks=true,
    linkcolor=blue,
    citecolor=blue,
    urlcolor=blue,
    filecolor=blue,
    linktoc=all
]{hyperref}

\usepackage[nameinlink]{cleveref}

\crefname{section}{Section}{Sections}
\crefname{theorem}{Theorem}{Theorems}
\crefname{lemma}{Lemma}{Lemmas}
\crefname{proposition}{Proposition}{Propositions}
\crefname{corollary}{Corollary}{Corollaries}
\crefname{definition}{Definition}{Definitions}
\crefname{example}{Example}{Examples}
\crefname{problem}{Problem}{Problems}
\crefname{remark}{Remark}{Remarks}

\newcommand{\R}{\mathbb{R}}
\newcommand{\N}{\mathbb{N}}
\newcommand{\Z}{\mathbb{Z}}
\newcommand{\T}{\mathbb{T}}
\newcommand{\Sp}{\mathbb{S}}
\newcommand{\eps}{\varepsilon}

\title[Mixing and enhanced dissipation in any spatial dimension]{Optimal mixing and enhanced dissipation for a smooth velocity field in any spatial dimension}
\author{Luca Melzi}
\address{Department of Mathematics, Imperial College London}
\email{l.melzi24@imperial.ac.uk}
\date{\today}


\begin{document}

\begin{abstract}
    In this note, we consider the advection-diffusion equation on the $d$-dimensional torus, where $d$ is arbitrary.
    First, for the associated transport equation, we provide an explicit example of a space-time smooth, divergence-free, random velocity field that is a universal exponential mixer.
    Second, we set the molecular diffusivity to a (small) positive value, and we show that the same velocity field is dissipation enhancing with optimal rate.
    Our construction is based on the Random Dynamical Systems viewpoint, and relies on the uniform ergodicity of the two-point Markov process associated with the two-dimensional Pierrehumbert model.
\end{abstract}

\maketitle
%\tableofcontents

\section{Introduction}
Let us consider the advection-diffusion equation
\begin{align}
    \partial_t\rho+u\cdot\nabla\rho=\kappa\Delta\rho\quad&\text{on }[0,\infty)\times\T^d, \label{eq:passive_scalar} \\
    \rho|_{t=0}=\rho_0\quad&\text{on }\T^d. \label{eq:passive_scalar_ic}
\end{align}
Here, $d\in\N,d\ge2$ is any spatial dimension, $\kappa\ge0$ is the molecular diffusivity, and $u:[0,\infty)\times\T^d\to\R^d$ is a given time-dependent divergence-free velocity field.
% , i.e.,
% \begin{equation*}
%     \nabla\cdot u=0\quad\text{on }[0,\infty)\times\T^d.
% \end{equation*}
We choose $\rho_0$ to be mean-free, which implies that the corresponding solution $\rho(t)$ is mean-free for any $t\ge0$.
Observe that testing \eqref{eq:passive_scalar} by $\rho$ itself yields the energy balance
\begin{equation}
    \|\rho(t)\|^2_{L^2}+2\kappa\int_0^t\|\nabla\rho(s)\|_{L^2}^2\;\mathrm{d}s=\|\rho_0\|^2_{L^2}\quad\forall t\ge0.
    \label{eq:energy_balance}
\end{equation}

\subsection{Mixing}

To begin with, we set $\kappa=0$ and study the phenomenon of \textit{mixing}, that is related to the formation of small scales induced by the dynamics \eqref{eq:passive_scalar}.
This is observed both numerically and experimentally, and in mathematical terms it can be described through the decay of some suitable norm of the solution $\rho$.
Setting $\kappa=0$ in \eqref{eq:energy_balance}, we see that the $L^2$-norm of $\rho$ is conserved.
In contrast, negative Sobolev norms may exhibit some decay.
\begin{definition} \label{def:univ_exp_mix}
    Let $s>0$.
    We say that a velocity field $u:[0,\infty)\times\T^d\to\R^d$ is a universal mixer if there exists $r:[0,\infty)\to[0,\infty)$ monotonically decreasing such that $\lim_{t\to\infty}r(t)=0$ and for any initial datum $\rho_0\in H^s$ it holds
    \begin{equation*}
        \|\rho(t)\|_{H^{-s}}\le r(t)\|\rho_0\|_{H^s}\quad\forall t\ge0.
    \end{equation*}
\end{definition}
If $r(t)=C\mathrm{e}^{-\gamma t}$ for some constants $C,\gamma>0$, we say that $u$ is \textit{exponentially} mixing.
For a Lipschitz continuous velocity field $u$, this is the fastest mixing rate achievable \cite{CrippaDeLellis2008,IyerKiselevXu2014,Seis2013}.
One of the aims of this work is to exhibit a universal exponential mixer on $\T^d$, and to do so we rely on randomness.
For $d=2$, one relevant example of a random universal exponential mixer is provided by the Pierrehumbert model, that was introduced in \cite{Pie1994} and then proved to be exponentially mixing with probability 1 in \cite{BluCotGva2023}.
The Pierrehumbert velocity field $U_{\underline{\omega}}$ is defined as follows.
For any $n\in\N$, let us sample the vector $\omega_n=(\omega_n^1,\omega_n^2)$ from a uniform distribution over $[0,2\pi)^2$, and denote by $\underline{\omega}=\{\omega_n\}_{n\in\N}$ the sequence of independent identically distributed phase shifts.
On the time interval $[n,n+1)$ we have
\begin{equation}
    U_{\underline{\omega}}(t,x_1,x_2):=\begin{cases}
        \begin{bmatrix}
            \sin(x_2-\omega_n^1)\\
            0
        \end{bmatrix}\quad&\text{if }t\in[n,n+\frac{1}{2}),\\
        \\
        \begin{bmatrix}
            0\\
            \sin(x_1-\omega_n^2)
        \end{bmatrix}\quad&\text{if }t\in[n+\frac{1}{2},n+1).
    \end{cases}
    \label{eq:Pierrehumbert}
\end{equation}
Note that if one randomises the amplitude of each shear instead of the phase, the resulting velocity field is still exponentially mixing on $\T^2$ \cite{coooperman2023}.
Other examples of universal exponential mixers for $d=2$ have been obtained as solutions to suitably forced stochastic Navier--Stokes equations \cite{coopermanRowan2026,BedBluPun2022}, while for the deterministic setting we mention the work \cite{ElgLisMat2025}.

Moving up to dimension $d=3$, the problem of providing an explicit example of a universal exponential mixer has been already addressed in the literature.
For instance, the work \cite{CotNav2026} proves that suitably randomised ABC flows are exponentially mixing.
Here, we further extend the analysis to any $d$, $d\ge3$.
To the best of our knowledge, this aim has only been addressed in \cite{ELGINDI2019106807}, where the authors provide an example of a deterministic, time-periodic exponential universal mixer with Sobolev space regularity $W^{s,p}$ for $s<\frac{1+\sqrt{5}}{2}$.
In contrast, in the present work we exhibit a random, space-time smooth exponential universal mixer.
More precisely, for some smooth function $g:\T^2\to\R^{d-2}$ and a sufficiently large fixed $N\in\N$, we define our velocity field on the time interval $[n(N+1),n(N+1)+N+1)$, $n\in\N$ as follows:
\begin{equation}
    u_{\underline{\omega}}(t,x_1,\dots,x_d):=\begin{cases}
        \begin{bmatrix}
            \boldsymbol{0}_2\\
            g(x_1,x_2)
        \end{bmatrix}\quad&\text{if }t\in[n(N+1),n(N+1)+1),\\
        \\
        \begin{bmatrix}
            U_{\underline{\omega}}(t,x_1,x_2)\\
            \boldsymbol{0}_{d-2}
        \end{bmatrix}\quad&\text{if }t\in[n(N+1)+1,n(N+1)+N+1),
    \end{cases}
    \label{eq:def_of_u}
\end{equation}
where $\boldsymbol{0}_m$ denotes the $m$-dimensional zero-vector.
\begin{remark}
    The velocity field defined in \eqref{eq:def_of_u} is smooth in space and piecewise constant in time.
    To obtain a space-time smooth velocity field, it is enough to multiply by a compactly supported smooth function on each time interval where $u_{\underline{\omega}}$ is autonomous.
    Namely, for any $m\in\N$, since the velocity field $u_{\underline{\omega}}$ is autonomous on the time interval $\left[\frac m2,\frac{m+1}2\right]$, define the smoothened version $\tilde u_{\underline{\omega}}\in\mathcal{C}^\infty([0,\infty)\times\T^d)$ by $\tilde u_{\underline{\omega}}(t,\cdot):=\beta(t)u_{\underline{\omega}}(\cdot)$, where
    \begin{equation*}
        \beta\in\mathcal C^\infty_c(\R),\quad\operatorname{supp}\beta\subset\left[\frac m2,\frac{m+1}2\right],\quad\beta\ge0,\quad\int_{\frac m2}^{\frac{m+1}2}\beta(s)\dd s=\frac12.
    \end{equation*}
    For the ease of presentation, we will outline our argument considering the piecewise constant in time velocity field $u_{\underline{\omega}}$, but the properties of $\beta$ ensure that the same argument applies to the smoothened version $\tilde u_{\underline{\omega}}$.
    
    Note that, for $\kappa=0$, smoothing in time yields the exact same dynamics \eqref{eq:passive_scalar} for the same initial condition \eqref{eq:passive_scalar_ic}, up to a suitable rescaling in time, see the Remark on page 1914 of \cite{YaoZla2017}.
\end{remark}
The non-degeneracy condition that we impose on $g$ is the following:
\begin{equation}
    \exists r\in\N\,:\,\operatorname{Span}\biggl\{\partial^{\alpha}g(x)\,\bigg|\,1\le|\alpha|\le r,\,\alpha\in\N^2\biggr\}=\R^{d-2}\quad\forall x\in\T^2,
    \label{eq:assumption_on_g}
\end{equation}
employing the notation $|\alpha|:=\alpha_1+\alpha_2$.
In Section \ref{subsec:choice_of_g} we provide an example of a smooth function $g$ fulfilling \eqref{eq:assumption_on_g}, see Example \ref{example:g}.
Let us recall that a solution $\rho(t)$ to \eqref{eq:passive_scalar}-\eqref{eq:passive_scalar_ic} with $u=u_{\underline{\omega}}$ given in \eqref{eq:def_of_u} fulfils
% \begin{equation}
%     \rho(n(N+1)+N+1)=\mathbb{E}_W[\rho_0\circ (T_\kappa^{n+1})^{-1}]
%     \label{eq:random_on_0}
% \end{equation}
\begin{equation}
    \rho(n(N+1))=\mathbb{E}_W[\rho_0\circ (T_\kappa^{n})^{-1}]
    \label{eq:random_on_0}
\end{equation}
for some random map $T_\kappa^{n}:\T^d\to\T^d$, with $\mathbb{E}_W$ denoting the expected value with respect to the Brownian motion (see Section \ref{sec:preliminaries} for the precise definitions).
Our main result is the following.
\begin{theorem} \label{thm:main}
    Fix any $q,s>0$.
    For every sufficiently small $\kappa\ge0$, there exist a positive random variable $D_\kappa$ with $\mathbb{E}|D_\kappa|^q<\infty$ uniformly in $\kappa$, and a constant $\gamma>0$ independent of $\kappa$ with the following property.
    For any mean-free $\varphi,\psi\in H^s(\T^d)$, the bound
    \begin{equation}
        \left|\int_{\T^d}\varphi(x)\psi(T_\kappa^n(x))\;\mathrm{d}x\right|\le D_\kappa\mathrm{e}^{-\gamma n}\|\varphi\|_{H^s}\|\psi\|_{H^s}
        \label{eq:decay_of_correlations_main_thm}
    \end{equation}
    holds almost surely for any $n\in\N$.
\end{theorem}
% \textcolor{red}{[Problem: $D_\kappa$ in \cite{CooIyeSon2025}'s theorem is already independent of $W$ but they don't prove it. There is an additional issue: averaging in $W$ to get the $W$-independent random variable would need a justification for the momentum bound.]}
Note that $D_\kappa>0$ and $T_\kappa^n(x)\in\T^d$ are random variables that depend both on the realisation of the Brownian motion and on the random (independent and uniformly distributed) phase shifts of the Pierrehumbert velocity field \eqref{eq:Pierrehumbert}.
When $\kappa=0$, a standard duality argument allows to infer exponential mixing.
\begin{corollary} \label{cor:exp_mixer}
    Let $\kappa=0$.
    Then, the velocity field defined in \eqref{eq:def_of_u} is a universal exponential mixer in the sense of Definition \ref{def:univ_exp_mix} with probability 1.
\end{corollary}

\subsection{Enhanced dissipation}
Let us now consider a sufficiently small (but positive) fixed value of the molecular diffusivity $\kappa$.
In this context, a related problem is enhanced dissipation, that is the diffusive counterpart of mixing.
In view of \eqref{eq:energy_balance}, the Poincaré inequality and the Gr\"onwall lemma give
\begin{equation*}
    \|\rho(t)\|_{L^2}\le\mathrm{e}^{-\kappa t}\|\rho_0\|_{L^2}.
\end{equation*}
Therefore, the characteristic time scale for the $L^2$-convergence of $\rho(t)$ to the mean value of $\rho_0$ is $T(\kappa)\sim\kappa^{-1}$.
Since the velocity field \eqref{eq:def_of_u} is a universal exponential mixer (cf. Corollary \ref{cor:exp_mixer}), in view of \cite[Theorem 2.5]{CotDelElg2020} the dissipative time scale can be enhanced to $T(\kappa)\sim|\log\kappa|^2$.
However, the optimal time scale for a Lipschitz continuous velocity field is $T(\kappa)\sim|\log\kappa|$ \cite{BedBluPun2021}, which we are able to obtain as a corollary of Theorem \ref{thm:main}.
\begin{corollary} \label{cor:enhanced_dissipation}
    Fix any $q,\varepsilon>0$.
    For every sufficiently small $\kappa>0$, there exist a positive random variable $D_\kappa$ with $\mathbb{E}|D_\kappa|^q<\infty$ uniformly in $\kappa$, and a constant $\gamma>0$ independent of $\kappa$ with the following property.
    There exists $t_0>0$ such that solution $\rho(t)$ of \eqref{eq:passive_scalar}-\eqref{eq:passive_scalar_ic} fulfils the almost-sure estimate
    \begin{equation*}
        \|\rho(t)\|_{L^\infty}\le\frac{D_\kappa}{\kappa^{\frac{d}{2}+\varepsilon}}\mathrm{e}^{-\gamma t}\|\rho_0\|_{L^1}\quad\forall t\ge t_0.
    \end{equation*}
\end{corollary}
Note that here $D_\kappa$ is a random variable that depends on the realisation of the velocity field \eqref{eq:def_of_u} only.
The optimal time scale $T(\kappa)\sim|\log\kappa|$ is then obtained by imposing $\mathrm{e}^{-\gamma T(\kappa)}\sim\kappa^{\frac{d}{2}+\eps}$.
It is also worth mentioning that the exponential decay rate becomes independent of diffusivity for large times.

\subsection{Discussion and organisation of the paper}

Theorem \ref{thm:main} asserts the decay of the correlations
\begin{equation*}
    \operatorname{Cor}_\kappa^n(\varphi,\psi):=\left|\int_{\T^d}\varphi(x)\psi(T_\kappa^n(x))\;\mathrm{d}x\right|,\quad\varphi,\psi\in L^2(\T^d).
\end{equation*}
To show such (exponential) decay \eqref{eq:decay_of_correlations_main_thm}, a sufficient condition is the uniform geometric ergodicity of the two-point Markov process associated with the map $T_\kappa^n$.
This idea first appeared in \cite{DolKalKor2004}, and was then employed in several other works, stemming from \cite{BluCotGva2023,BedBluPun2022,BedBluPun2021}.
In our setting, we refrain from showing that two-point Markov process associated with the map $T_\kappa^n$ is uniformly geometrically ergodic.
Instead, we solely rely on the uniform geometric ergodicity of the two-point Markov process associated with the two-dimensional Pierrehumbert model.
The decay of correlations on two-dimensional slices is then propagated to higher dimensions by means of a suitably defined contraction map.
\begin{remark}
    To prove that a given $d$-dimensional velocity field induces exponential decay of correlations, the standard approach is to use the machinery developed by Blumenthal--Coti Zelati--Gvalani \cite{BluCotGva2023}.
    In particular, when applying the Furstenberg criterion, one has to check that the Jacobian matrix of the transport map with respect to the noise space is full-rank and that its restriction to noise perturbations that leave the endpoint unchanged to first order is surjective.
    This also ensures that the small set property holds in the setting of the Harris theorem.
    Even with a suitably large noise space, the number of columns of the Jacobian matrix scales quadratically with the space dimension $d$ (while the number of rows scales linearly), making it increasingly costly to check the full-rank condition.
    We emphasise that our construction with the velocity field \eqref{eq:def_of_u} described above allows us to circumvent such problem.
\end{remark}
The paper is organised as follows.
In Section \ref{sec:preliminaries}, we set the notation and collect all the results that will be needed to prove Theorem \ref{thm:main} and Corollaries \ref{cor:exp_mixer} and \ref{cor:enhanced_dissipation}.
More specifically, we devote Section \ref{subsec:choice_of_g} to the choice of the shear $(\boldsymbol{0}_2,g)$ in \eqref{eq:def_of_u}, stating the fundamental property Lemma \ref{lemma:unif_bdd_from_1} and providing an explicit Example \ref{example:g}.
Next, in Section \ref{subsec:2d} we define a suitable norm $\|\cdot\|$ for functions on $\T^2\times\T^2$ and recall the fundamental result of \cite{CooIyeSon2025} for uniform-in-$\kappa$ ergodicity in the Pierrehumbert model.
We then prove in Section \ref{subsec:contraction_map} that a suitably defined map is a contraction with respect to $\|\cdot\|$.
Finally, we devote Section \ref{sec:proofs} to the proof of the uniform-in-$\kappa$ decay of correlations \eqref{eq:decay_of_correlations_main_thm} in Theorem \ref{thm:main} (whence Corollary \ref{cor:exp_mixer} immediately follows), and to the proof of the enhanced dissipation result in Corollary \ref{cor:enhanced_dissipation}.
This is done in Sections \ref{subsec:mixing} and \ref{subsec:enh_diss} respectively.

\section{Preliminaries} \label{sec:preliminaries}

Consider two probability spaces $(\Omega_W,\mathcal{F}_W,\mathbb{P}_W)$ and $(\Omega_0,\mathcal{F}_0,\mathbb{P}_0)$, with the product space defined as
\begin{equation*}
    \Omega=\Omega_W\times\Omega_0,\quad\mathcal{F}=\mathcal{F}_W\otimes\mathcal{F}_0,\quad\mathbb{P}=\mathbb{P}_W\otimes\mathbb{P}_0.
\end{equation*}
Correspondingly, we will denote by $\mathbb{E}_W$, $\mathbb{E}_0$ and $\mathbb{E}$ the expected values.
We choose $\Omega_0=([0,2\pi)^2)^\N$ with $\mathbb{P}_0$ being the uniform probability distribution.
Moreover, let $W_t$ denote a $\T^d$-valued Brownian motion on the space $(\Omega_W,\mathcal{F}_W,\mathbb{P}_W)$.
For $\kappa\ge0$, $n\in\N$ and $\ell\in\{1,\dots,N\}$, let the random variables $X_t^{\kappa,n,\ell},Y_t^{\kappa,n}:\T^d\to\T^d$ on $(\Omega,\mathcal{F},\mathbb{P})$ solve the SDEs
\begin{align*}
    \mathrm{d}Y_t^{\kappa,n}(x)=&u_{\underline{\omega}}\left(n(N+1)+t,Y_t^{\kappa,n}(x)\right)\mathrm{d}t+\sqrt{2\kappa}\mathrm{d}W_{n(N+1)+t},&Y_0^{\kappa,n}(x)=x,\quad&t\in[0,1],\\
    \mathrm{d}X_t^{\kappa,n,\ell}(x)=&u_{\underline{\omega}}\left(n(N+1)+\ell+t,X_t^{\kappa,n,\ell}(x)\right)\mathrm{d}t+\sqrt{2\kappa}\mathrm{d}W_{n(N+1)+\ell+t},&X_0^{\kappa,n,\ell}(x)=x,\quad&t\in[0,1],
\end{align*}
with $u_{\underline{\omega}}$ given in \eqref{eq:def_of_u}.
We define now
\begin{equation*}
    Z_\kappa^n:=X_1^{\kappa,n,N}\circ\cdots\circ X_1^{\kappa,n,1}\circ Y_{1}^{\kappa,n},\qquad T_\kappa^0:=\operatorname{Id},\qquad  T_\kappa^{n+1}:=Z_\kappa^n\circ Z_\kappa^{n-1}\circ\cdots\circ Z_\kappa^0,
\end{equation*}
so that \eqref{eq:random_on_0} holds.

\subsection{Choice of the shear} \label{subsec:choice_of_g}

In this section, we are concerned with the condition \eqref{eq:assumption_on_g} that the function $g:\T^2\to\R^{d-2}$ in \eqref{eq:def_of_u} needs to satisfy.
The main reason for such an assumption is that we want to apply the following lemma.
\begin{lemma} \label{lemma:unif_bdd_from_1}
    Let $m,p,r\in\N$.
    Consider a smooth function $g:\T^m\to\R^p$ satisfying
    \begin{equation}
        \operatorname{Span}\biggl\{\partial^{\alpha}g(x)\,\bigg|\,1\le|\alpha|\le r,\,\alpha\in\N^m\biggr\}=\R^p\quad\forall x\in\T^m.
        \label{eq:assumption_on_g_mn}
    \end{equation}
    Then, there exist $\varepsilon,\delta>0$ such that for any smooth $f:\T^m\to\R^p$ satisfying $\|f-g\|_{\mathcal{C}^{r}}\le\delta$ it holds
    \begin{equation*}
        \frac{1}{(2\pi)^m}\left|\int_{\T^m}\mathrm{e}^{\mathrm{i}k\cdot f(x)}\;\mathrm{d}x\right|\le1-\varepsilon\quad\forall k\in\Z_0^{p}:=\Z^{p}\setminus\{\boldsymbol{0}_{p}\}.
    \end{equation*}
\end{lemma}
\begin{proof}
    Since for any $k\in\Z_0^{p}$ there exists $\varepsilon_k>0$ such that
    \begin{equation*}
        \frac{1}{(2\pi)^m}\left|\int_{\T^m}\mathrm{e}^{\mathrm{i}k\cdot f(x)}\;\mathrm{d}x\right|\le1-\varepsilon_k,
    \end{equation*}
    it is enough to show that
    \begin{equation}
        \left|\int_{\T^m}\mathrm{e}^{\mathrm{i}k\cdot f(x)}\;\mathrm{d}x\right|\to0\quad\text{as }|k|\to\infty.
        \label{eq:uniform_step_m1}
    \end{equation}
    We present here the proof in the simplified setting $m=1$, i.e., with $f,g$ depending on one variable $x\in\T$ only, which is the case in Example \ref{example:g} below.
    Then, we will briefly mention how to extend the proof to any $m\in\N$, see Remark \ref{rmk:m_generic}.
    
    Let us first prove the claim \eqref{eq:uniform_step_m1} for $f\equiv g$.
    In view of \eqref{eq:assumption_on_g_mn}, for any $x\in\T$ and any $\xi\in\Sp^{p-1}$, there exists $q\in\{1,\dots,r\}$ such that
    \begin{equation*}
        \left|\xi\cdot \frac{\mathrm{d}^qg}{\mathrm{d}x^q}(x)\right|=\left|\frac{\mathrm{d}^q}{\mathrm{d}x^q}(\xi\cdot g(x))\right|>0.
    \end{equation*}
    Since $g$ is smooth, for any $(\bar x,\bar\xi)\in\T\times\Sp^{p-1}$ there exist a constant $c>0$, an index $q\in\{1,\dots,r\}$, and an open neighbourhood $\Gamma$ of $(\bar x,\bar\xi)$ such that
    \begin{equation*}
        \left|\frac{\mathrm{d}^{q}}{\mathrm{d}x^{q}}(\xi\cdot g(x))\right|\ge c\quad\forall(x,\xi)\in \Gamma.
    \end{equation*}
    Since $\T\times\Sp^{p-1}$ is compact, finitely many of such neighbourhoods $\Gamma$ form a finite open cover of $\T\times\Sp^{p-1}$.
    Namely, for some $J\in\N$ depending on $g$ only, there exist constants $\{c_j\}_{j=1}^J\subset(0,\infty)$, indices $\{q_j\}_{j=1}^J\subset\{1,\dots,r\}$, and a finite open cover $\{\Gamma_j\subset\T\times\Sp^{p-1}\}_{j=1}^J$ such that
    \begin{equation*}
        \left|\frac{\mathrm{d}^{q_j}}{\mathrm{d}x^{q_j}}(\xi\cdot g(x))\right|\ge c_j\quad\forall(x,\xi)\in \Gamma_j.
    \end{equation*}
    Up to increasing $J$, it is then possible to choose finite open covers $\{U_j\}_{j=1}^J,\{V_j\}_{j=1}^J$ of $\T$ and $\Sp^{p-1}$ respectively such that
    % $\Gamma_j\subset U_j\times V_j$ for any $j$.
    %  and it holds (up to re-labelling)
    \begin{equation}
        \left|\frac{\mathrm{d}^{q_j}}{\mathrm{d}x^{q_j}}(\xi\cdot g(x))\right|\ge c_j\quad\forall(x,\xi)\in U_j\times V_j.
        \label{eq:uniform_step_0}
    \end{equation}
    Clearly, $\{U_j\times V_j\}_{j=1}^J$ is then a finite cover of $\T\times\Sp^{p-1}$.
    Consider now the partition of unity $\left\{\rho_j:\T\times\Sp^{p-1}\to[0,1]\right\}_{j=1}^J$ subordinated to the cover $\{U_j\times V_j\}_{j=1}^J$, i.e., $\operatorname{supp}\rho_j\subset U_j\times V_j$ for any $j$.
    Note that, for any $\bar\xi\in\Sp^{p-1}$ fixed, $\left\{\rho_j(\cdot,\bar\xi):\T\to[0,1]\right\}_{j=1}^J$ is a partition of unity subordinated to the cover $\{U_j\}_{j=1}^J$ of $\T$.
    In particular, $\operatorname{supp}\rho_j(\cdot,\bar\xi)$ is compact.
    For any $\lambda\ge1$ and any $\xi\in\Sp^{p-1}$, consider the integral
    \begin{equation*}
        I_\xi(\lambda):=\int_{\T}\mathrm{e}^{\mathrm{i}\lambda\xi\cdot g(x)}\;\mathrm{d}x=\sum_{j=1}^J\int_{U_j}\mathrm{e}^{\mathrm{i}\lambda\xi\cdot g(x)}\rho_j(x,\xi)\;\mathrm{d}x.
    \end{equation*}
    We claim that for any $j$ there exists a constant $C_j>0$ such that
    \begin{equation}
        \left|\int_{U_j}\mathrm{e}^{\mathrm{i}\lambda\xi\cdot g(x)}\rho_j(x,\xi)\;\mathrm{d}x\right|\le C_j\lambda^{-1/q_j}\quad\forall\lambda\ge1,\,\forall\xi\in\Sp^{p-1}.
        \label{eq:uniform_step_1}
    \end{equation}
    As a consequence, denoting by $|\cdot|$ the Lebesgue measure of a set and $C:=\max_jC_j>0$, we get
    \begin{equation*}
        |I_\xi(\lambda)|\le C\sum_{j=1}^J\left|U_j\right|\lambda^{-1/q_j}\le2\pi CJ\lambda^{-1/r}.
    \end{equation*}
    Taking now $\lambda=|k|$ and $\xi=k/|k|$ yields the thesis.
    Therefore, it remains to prove the claim \eqref{eq:uniform_step_1}.
    To this end, we consider two different cases: $q_j=1$ and $q_j\ge2$.

    In the latter case $q_j\ge2$, together with \eqref{eq:uniform_step_0}, the Corollary on page 334 of \cite{Ste1993} gives the existence of a constant $K_j$ independent of $\lambda,\xi$ such that
    \begin{equation*}
        \left|\int_{U_j}\mathrm{e}^{\mathrm{i}\lambda\xi\cdot g(x)}\rho_j(x,\xi)\;\mathrm{d}x\right|\le K_j|c_j\lambda|^{-1/q_j}\int_{U_j}\left|\partial_x\rho_j(x,\xi)\right|\;\mathrm{d}x.
    \end{equation*}
    Since $\rho_j(\cdot,\xi)$ is a smooth function and its support is a compact subset of $U_j$, \eqref{eq:uniform_step_1} follows.
    
    Let us now address the case $q_j=1$, for which an integration by parts yields
    \begin{equation*}
        \int_{U_j}\mathrm{e}^{\mathrm{i}\lambda\xi\cdot g(x)}\rho_j(x,\xi)\;\mathrm{d}x=\frac{\mathrm{i}}{\lambda}\int_{U_j}\mathrm{e}^{\mathrm{i}\lambda\xi\cdot g(x)}\partial_x\left(\frac{\rho_j(x,\xi)}{\partial_x(\xi\cdot g(x))}\right)\;\mathrm{d}x,
    \end{equation*}
    whence the estimate
    \begin{equation*}
        \left|\int_{U_j}\mathrm{e}^{\mathrm{i}\lambda\xi\cdot g(x)}\rho_j(x,\xi)\;\mathrm{d}x\right|\le2\pi\lambda^{-1}\max_{x\in\operatorname{supp}\rho_j(\cdot,\xi)}\left|\partial_x\left(\frac{\rho_j(x,\xi)}{\partial_x(\xi\cdot g(x))}\right)\right|.
    \end{equation*}
    In view of \eqref{eq:uniform_step_0}, the function inside the absolute value is smooth and its support is a subset of $\operatorname{supp}\rho_j(\cdot,\xi)$.
    Therefore, \eqref{eq:uniform_step_1} follows for $q_j=1$ as well, concluding the proof in the case $f\equiv g$.
    
    Finally, if $f\not\equiv g$, it is enough to observe the following.
    Together with \eqref{eq:uniform_step_0}, the uniform bound $\|f-g\|_{\mathcal{C}^{r}}\le\delta$ implies
    \begin{equation}
        \left|\frac{\mathrm{d}^{q_j}}{\mathrm{d}x^{q_j}}(\xi\cdot f(x))\right|\ge\frac{c_j}{2}\quad\forall(x,\xi)\in U_j\times V_j,
        \label{eq:uniform_step_2}
    \end{equation}
    up to choosing $\delta>0$ small enough.
    Here, $c_j$ is the same constant appearing in \eqref{eq:uniform_step_0}, thus it depends on $g$ only, and is independent of $f$.
    Once we have \eqref{eq:uniform_step_2}, the same argument above applies with $f$ in place of $g$.
    This concludes the proof of Lemma \ref{lemma:unif_bdd_from_1}.
\end{proof}
\begin{remark}[Generalisation to any $m\in\N$] \label{rmk:m_generic}
    We briefly sketch here how to extend the proof of Lemma \ref{lemma:unif_bdd_from_1} to any $m\in\N$.
    If $m\ge2$, for any $j=1,\dots,J$ there exist a direction $v_j\in\Sp^{m-1}$ and a differentiation order $q_j\in\N$, $1\le q_j\le r$, such that
    \begin{equation*}
        |(v_j\cdot\nabla)^{q_j}\xi\cdot g(x)|\ge c_j\quad\forall(x,\xi)\in U_j\times V_j.
    \end{equation*}
    Indeed, for any smooth function $h:\R^m\to\R^p$ we have the pointwise estimate (see \cite{Ban1938,Har1996})
    \begin{align*}
        \max_{\substack{v\in\Sp^{m-1}}}|(v\cdot\nabla)^qh|\ge\max_{\substack{|\alpha|=q\\\alpha\in\N^m}}|\partial^\alpha h|\quad\forall q\in\N.
    \end{align*}
    It is then possible to choose $U_j\subset\T^m$ small enough in such a way that there exists a smooth change of coordinates from $x\in U_j$ to $(y,y')\in O_j\times O_j'\subset[-1,1]\times[-1,1]^{m-1}$, where $O_j$ is an open connected interval and $O_j'$ is the Cartesian product of $m-1$ open connected intervals.
    Here, $y$ denotes the coordinate in the direction $v_j\in\Sp^{m-1}$, so that $x=x_j+yv_j+M_jy'$ for some $x_j\in U_j$ and some matrix $M_j\in\R^{m\times(m-1)}$ whose columns are unit vectors such that
    \begin{equation*}
        \operatorname{Span}\{v_j\}\oplus\operatorname{col}(M_j)=\R^m,
    \end{equation*}
    where $\operatorname{col}(\cdot)$ denotes the range of a linear application.
    Therefore, for any $j$ there exists a smooth, compactly supported function $a_j:O_j\times O_j'\times\Sp^{p-1}\to\R$ such that
    \begin{equation*}
        I_\xi(\lambda)=\sum_{j=1}^J\int_{O_j'}\biggl(\int_{O_j}\mathrm{e}^{\mathrm{i}\lambda\xi\cdot g(x_j+yv_j+M_jy')}a_j(y,y',\xi)\;\mathrm{d}y\biggr)\mathrm{d}y',
    \end{equation*}
    appealing to Fubini's theorem.
    The same argument in the proof of Lemma \ref{lemma:unif_bdd_from_1} then applies to the one-dimensional integral in the scalar variable $y$.
\end{remark}
We now conclude the section by providing an example of a function $g:\T^2\to\R^{d-2}$ that fulfils \eqref{eq:assumption_on_g}.
In particular, our function $g$ depends on the first variable $x_1$ only, as in the simplified setting of the proof of Lemma \ref{lemma:unif_bdd_from_1}.
\begin{example} \label{example:g}
    Consider $g(x_1,x_2)=(\sin x_1,\sin(2x_1),\dots,\sin((d-2)x_1))$.
    Then, \eqref{eq:assumption_on_g} holds with $r=2(d-2)$.

    Indeed, by contradiction assume that for some $\bar{x}\in\T$ there exists $a\in\R^{d-2}\setminus\{\boldsymbol{0}_{d-2}\}$ such that for every $j\in\{1,\dots,2(d-2)\}$ it holds
    \begin{equation*}
        a\cdot g^{(j)}(\bar{x})=0,
    \end{equation*}
    where, with a slight abuse of notation we still denote by $g$ the function of one variable $g(x)=(\sin x,\dots,\sin((d-2)x))$.
    Therefore, the function
    \begin{equation*}
        F(x):=\sum_{\ell=1}^{d-2}a_\ell\ell\cos(\ell x),\quad x\in\T
    \end{equation*}
    satisfies
    \begin{equation}
        F(\bar{x})=F'(\bar{x})=F''(\bar{x})=\ldots=F^{(2(d-2)-1)}(\bar{x})=0.
        \label{eq:ics}
    \end{equation}
    Observe that $F$ also fulfils the $2(d-2)$-th order ODE
    \begin{equation*}
        \left(\frac{\mathrm{d}^2}{\mathrm{d}x^2}+1^2\right)\left(\frac{\mathrm{d}^2}{\mathrm{d}x^2}+2^2\right)\cdots\left(\frac{\mathrm{d}^2}{\mathrm{d}x^2}+(d-2)^2\right)F(x)=0\quad\forall x\in\T.
    \end{equation*}
    Together with the initial conditions \eqref{eq:ics}, this entails $F\equiv0$, that is $a=\boldsymbol{0}_{d-2}$.
\end{example}

\subsection{The two-dimensional case} \label{subsec:2d}
Our argument relies on known results for universal exponential mixers on $\T^2$.
For this reason, the first two coordinates $(x_1,x_2)$ of $\T^d$ will play a different role compared to the last $d-2$.
Bearing this in mind, we recall in this section the fundamental ergodicity property of the two-dimensional Pierrehumbert model, and we introduce a norm for functions defined over $\T^2\times\T^2$ with respect to which the map $\mathcal{T}_k^\kappa$ defined in Section \ref{subsec:contraction_map} is a contraction.

Let $P_{\kappa,n,\ell}^{(2)}$ denote the two-point transition kernel associated to one step of the Pierrehumbert velocity field \eqref{eq:Pierrehumbert} on the time interval $[n(N+1)+\ell,n(N+1)+\ell+1)$ for $\ell\in\{1,\dots,N\}$.
Namely, $P_{\kappa,n,\ell}^{(2)}$ acts on continuous bounded functions $\eta\in\mathcal{C}^0_B(\mathbb{T}^2\times\mathbb{T}^2;\mathbb{C})$ as
\begin{equation*}
    \eta\mapsto P_{\kappa,n,\ell}^{(2)}\eta:=\mathbb{E}\biggl[\eta\circ\biggl(\begin{bmatrix}
        (X_{1}^{\kappa,n,\ell})_1\\
        (X_{1}^{\kappa,n,\ell})_2
    \end{bmatrix},\begin{bmatrix}
        (X_{1}^{\kappa,n,\ell})_1\\
        (X_{1}^{\kappa,n,\ell})_2
    \end{bmatrix}\biggr)\biggr]\in\mathcal{C}^0_B(\mathbb{T}^2\times\mathbb{T}^2;\mathbb{C}).
\end{equation*}
% Here, the expectation in taken with respect to both the Brownian motion and the randomness in $u$, and we have employed the notation
% \begin{align*}
%     X_{1/N}^{\kappa,0}\times X_{1/N}^{\kappa,0}:\T^2\times\T^2&\to\T^2\times\T^2,\\
%     (x,y)&\mapsto X_{1/N}^{\kappa,0}\times X_{1/N}^{\kappa,0}(x,y)=\begin{bmatrix}
%         X_{1/N}^{\kappa,0}(x)\\
%         X_{1/N}^{\kappa,0}(y)
%     \end{bmatrix}.
% \end{align*}
Given a function $V:\T^2\times\T^2\setminus\{(x,x)\mid x\in\T^2\}\to[1,\infty)$, we define the norm
\begin{equation*}
    \|\eta\|_V:=\sup_{\substack{x,y\in\mathbb{T}^2\\x\ne y}}\frac{|\eta(x,y)|}{V(x,y)},\qquad\eta\in\mathcal{C}^0_B(\T^2\times\T^2;\mathbb{C}),
\end{equation*}
and given any $h:\mathbb{T}^2\to\mathbb{C}$ we define $h^{(2)}:\T^2\times\T^2\to\mathbb{C}$ by
\begin{equation*}
    h^{(2)}(x,y):=h(x)\overline{h(y)}.
\end{equation*}
Finally, we denote
\begin{equation*}
    (P^{(2)}_{\kappa,n})^\ell:=P^{(2)}_{\kappa,n,\ell}\circ\cdots\circ P^{(2)}_{\kappa,n,1}.
\end{equation*}
The building block of our argument is the $V$-uniform geometric ergodicity of $(P_{\kappa,n}^{(2)})^\ell$, which was proved in \cite{BluCotGva2023} for $\kappa=0$ and then extended uniformly to sufficiently small $\kappa\ge0$ in \cite[Lemma 3.2]{CooIyeSon2025}.
We report such result here for the reader's convenience.
\begin{proposition}[\cite{CooIyeSon2025}] \label{prop:BluCotGva2023}
    There exist constants $C,\gamma>0$ and a function $V\in\mathcal C^0(\T^2\times\T^2\setminus\{(x,x)\mid x\in\T^2\})\cap L^1(\mathbb{T}^2\times\mathbb{T}^2)$ such that $V\ge1$ almost everywhere with the following property.
    For any $n\in\N$, any $\ell\in\{1,\dots,N\}$, any mean-free function $\eta$ with $\|\eta\|_V<\infty$, and any $\kappa\ge0$ small enough, it holds
    \begin{equation*}
        \left\|(P_{\kappa,n}^{(2)})^\ell\eta\right\|_V\le C\mathrm{e}^{-\gamma\ell}\|\eta\|_V.
    \end{equation*}
\end{proposition}
In our applications, we will always choose $N$ in \eqref{eq:def_of_u} large enough in order to make the rate $C\mathrm{e}^{-\gamma N}$ arbitrarily small.
Here and in the following, we employ the notation $\|\cdot\|$ for the norm
\begin{equation*}
    \|\eta\|:=|\bar{\eta}|+\|\eta_{\neq}\|_V,
\end{equation*}
where $\bar{\eta}:=\fint_{\T^2\times\T^2}\eta=(2\pi)^{-4}\int_{\T^2\times\T^2}\eta$ is the average and $\eta_{\ne}:=\eta-\bar{\eta}$ is the oscillatory term.
In the following lemma, we prove some lower bounds for the norm $\|\cdot\|$ which will come in handy in the following.
\begin{lemma} \label{lemma:prop_of_norm}
    Let $\eta\in\mathcal{C}^0_B(\T^2\times\T^2;\mathbb{C})$.
    Then, the following chain of inequalities hold.
    \begin{equation}
        \frac{1}{\|V\|_{L^1}}\int_{\T^2\times\T^2}|\eta|\le\|\eta\|_V\le\|\eta\|.
        \label{eq:prop_of_norm}
    \end{equation}
\end{lemma}
\begin{proof}
    The first of the two inequalities in \eqref{eq:prop_of_norm} follows by an application of the H\"older inequality:
    \begin{equation*}
        \int_{\T^2\times\T^2}|\eta|=\int_{\T^2\times\T^2}\frac{|\eta|}{V}V\le\|V\|_{L^1}\sup_{\T^2\times\T^2\setminus\{(x,x)\mid x\in\T^2\}}\frac{|\eta|}{V}=\|V\|_{L^1}\|\eta\|_V.
    \end{equation*}
    Regarding the second inequality in \eqref{eq:prop_of_norm}, it is enough to apply the triangle inequality and argue as follows:
    \begin{equation*}
        \|\eta\|_V\le\|\bar{\eta}\|_V+\|\eta_{\ne}\|_V=|\bar{\eta}|\sup_{\T^2\times\T^2\setminus\{(x,x)\mid x\in\T^2\}}\frac{1}{V}+\|\eta_{\ne}\|_V\le|\bar{\eta}|+\|\eta_{\ne}\|_V=\|\eta\|,
    \end{equation*}
    where in the last inequality we have used the fact that $V\ge1$ almost everywhere.
    This shows the second inequality in \eqref{eq:prop_of_norm} and hence the proof is complete.
\end{proof}

\subsection{The contraction map} \label{subsec:contraction_map}
Let us employ the following notation.
Given the $\T^d$-valued stochastic process $W_t$, let $W^2_t=((W_t)_1,(W_t)_2)$ be the $\T^2$-valued stochastic process obtained as a projection onto the first two coordinates.
% \color{red}
% Denote by $W^{d-2}_t$ the remaining $d-2$ components.
% \color{black}
Fix $k\in\Z^{d-2}$ and define the map
\begin{equation*}
    \mathcal{T}^{\kappa,n}_k:\mathcal{C}^0_B(\T^2\times\T^2;\mathbb{C})\to\mathcal{C}^0_B(\T^2\times\T^2;\mathbb{C})
\end{equation*}
via
\begin{equation*}
    \mathcal{T}^{\kappa,n}_k[\eta]\begin{pmatrix}
        x\\
        y
    \end{pmatrix}=\mathbb{E}_W\mathrm{e}^{\mathrm{i}k\cdot\left(G_n^\kappa(x)-G_n^\kappa(y)\right)}(P_{\kappa,n}^{(2)})^N[\eta]\begin{pmatrix}
        x+\sqrt{2\kappa}\left(W^2_{n(N+1)+1}-W^2_{n(N+1)}\right)\\
        y+\sqrt{2\kappa}\left(W^2_{n(N+1)+1}-W^2_{n(N+1)}\right)
    \end{pmatrix},
\end{equation*}
where
\begin{equation*}
    G^n_\kappa(x):=\int_0^1g\left(x+\sqrt{2\kappa}\left(W^2_{n(N+1)+t}-W^2_{n(N+1)}\right)\right)\;\mathrm{d}t,\quad x\in\T^2.
\end{equation*}
Observe that when $\kappa=0$ we have $G^n_0\equiv g$ for any $n$.
Finally, let us denote
\begin{equation*}
    (\mathcal{T}^\kappa_k)^{n+1}=\mathcal{T}^{\kappa,n}_k\circ \mathcal{T}^{\kappa,n-1}_k\circ\dots\circ \mathcal{T}^\kappa_k.
\end{equation*}
In the following lemma, we show that the map $\mathcal{T}^\kappa_k$ is a contraction with rate independent of $\kappa$.
\begin{lemma} \label{lemma:contraction}
    Let $g$ fulfil \eqref{eq:assumption_on_g}, and let $V$ be given in Proposition \ref{prop:BluCotGva2023}.
    Then, there exist $C,\gamma>0$ such that for any $n\in\N$ and any $\kappa\ge0$ sufficiently small we have
    \begin{align*}
        \|(\mathcal{T}^\kappa_k)^n\eta\|\le C&\mathrm{e}^{-\gamma n}\|\eta\|\quad\forall k\in\Z^{d-2}_0,\,\forall\eta\in\mathcal{C}^0_B(\T^2\times\T^2;\mathbb{C}),\\
        \|(\mathcal{T}^\kappa_k)^n\eta\|\le C&\mathrm{e}^{-\gamma n}\|\eta\|\quad\forall k\in\Z^{d-2},\,\forall\eta\in\mathcal{C}^0_B(\T^2\times\T^2;\mathbb{C}),\,\fint_{\T^2\times\T^2}\eta=0.
    \end{align*}
\end{lemma}
\begin{proof}
    Let us begin with the case $k\ne0$.
    Fixing any $n\in\N$, due to linearity one has
    \begin{align}
    \begin{split}
        \mathcal{T}^{\kappa,n}_k[\eta](x,y)=&\mathbb{E}_W\mathrm{e}^{\mathrm{i}k\cdot\left(G^n_\kappa(x)-G^n_\kappa(y)\right)}\bar{\eta}\\
        &+\mathbb{E}_W\mathrm{e}^{\mathrm{i}k\cdot\left(G^n_\kappa(x)-G^n_\kappa(y)\right)}(P_{\kappa,n}^{(2)})^N[\eta_{\ne}]\begin{pmatrix}
        x+\sqrt{2\kappa}\left(W^2_{n(N+1)+1}-W^2_{n(N+1)}\right)\\
        y+\sqrt{2\kappa}\left(W^2_{n(N+1)+1}-W^2_{n(N+1)}\right)
    \end{pmatrix},
        \label{eq:contraction_step_m1}
    \end{split}
    \end{align}
    which in turn gives
    \begin{align}
    \begin{split}
        |\overline{\mathcal{T}^{\kappa,n}_k[\eta]}|\le&|\bar{\eta}|\mathbb{E}_W\left|\fint_{\mathbb{T}^2\times\mathbb{T}^2}\mathrm{e}^{\mathrm{i}k\cdot\left(G^n_\kappa(x)-G^n_\kappa(y)\right)}\;\mathrm{d}x\,\mathrm{d}y\right|\\
        &+\mathbb{E}_W\left|\fint_{\mathbb{T}^2\times\mathbb{T}^2}\mathrm{e}^{\mathrm{i}k\cdot\left(G^n_\kappa(x)-G^n_\kappa(y)\right)}(P_{\kappa,n}^{(2)})^N[\eta_{\ne}]\begin{pmatrix}
        x+\sqrt{2\kappa}\left(W^2_{n(N+1)+1}-W^2_{n(N+1)}\right)\\
        y+\sqrt{2\kappa}\left(W^2_{n(N+1)+1}-W^2_{n(N+1)}\right)
    \end{pmatrix}\;\mathrm{d}x\,\mathrm{d}y\right|.
        \label{eq:contraction_step_0}
    \end{split}
    \end{align}
    The first term on the right-hand side of \eqref{eq:contraction_step_0} can be handled by observing that
    \begin{equation}
        \left|\fint_{\mathbb{T}^2\times\mathbb{T}^2}\mathrm{e}^{\mathrm{i}k\cdot\left(G^n_\kappa(x)-G^n_\kappa(y)\right)}\;\mathrm{d}x\,\mathrm{d}y\right|=\left|\fint_{\mathbb{T}^2}\mathrm{e}^{\mathrm{i}k\cdot G^n_\kappa(x)}\;\mathrm{d}x\right|^2,
        \label{eq:contraction_step_01}
    \end{equation}
    where
    \begin{align}
    \begin{split}
        \mathbb{E}_W\left|\fint_{\mathbb{T}^2}\mathrm{e}^{\mathrm{i}k\cdot G^n_\kappa(x)}\;\mathrm{d}x\right|^2\le&\mathbb{E}_W\left[\left|\fint_{\mathbb{T}^2}\mathrm{e}^{\mathrm{i}k\cdot G^n_\kappa(x)}\;\mathrm{d}x\right|^2\,\bigg|\,|W^2_{n(N+1)+t}-W^2_{n(N+1)}|\le\kappa^{-1/4}\;\forall t\in[0,1]\right]\\
        &+\mathbb{P}_W\left(\max_{t\in[0,1]}|W^2_{n(N+1)+t}-W^2_{n(N+1)}|>\kappa^{-1/4}\right).
        \label{eq:contraction_step_02}
    \end{split}
    \end{align}
    In view of assumption \eqref{eq:assumption_on_g} and the fact that
    \begin{equation*}
        \left|\partial^\alpha G^n_\kappa(x)-\partial^\alpha g(x)\right|=\left|\int_0^1\left[\partial^\alpha g\left(x+\sqrt{2\kappa}\left(W^2_{n(N+1)+t}-W^2_{n(N+1)}\right)\right)-\partial^\alpha g(x)\right]\;\mathrm{d}t\right|\le\sqrt{2}\kappa^{1/4}\|g\|_{\mathcal{C}^{|\alpha|+1}}
    \end{equation*}
    for all $|W^2_{n(N+1)+t}-W^2_{n(N+1)}|\le\kappa^{-1/4}$ and any $\alpha\in\N^2$, we can apply Lemma \ref{lemma:unif_bdd_from_1} with $m=2$, $p=d-2$, and $f\equiv G^n_\kappa$.
    Therefore, there exists a constant $\varepsilon>0$ depending on $g$ only such that
    \begin{equation*}
        \left|\fint_{\mathbb{T}^2}\mathrm{e}^{\mathrm{i}k\cdot G^n_\kappa(x)}\;\mathrm{d}x\right|^2\le1-2\varepsilon
    \end{equation*}
    for all $|W^2_{n(N+1)+t}-W^2_{n(N+1)}|\le\kappa^{-1/4}$.
    Together with \eqref{eq:contraction_step_01} and \eqref{eq:contraction_step_02}, this implies
    \begin{equation}
        \mathbb{E}_W\left|\fint_{\mathbb{T}^2\times\mathbb{T}^2}\mathrm{e}^{\mathrm{i}k\cdot\left(G^n_\kappa(x)-G^n_\kappa(y)\right)}\;\mathrm{d}x\,\mathrm{d}y\right|\le1-2\varepsilon+\mathbb{P}_W\left(\max_{t\in[0,1]}|W^2_{n(N+1)+t}-W^2_{n(N+1)}|>\kappa^{-1/4}\right)\le1-\varepsilon,
        \label{eq:contraction_step_1}
    \end{equation}
    where it is always possible to choose $\kappa$ small enough so that the last inequality holds, as $\varepsilon$ is independent of $\kappa$.
    Regarding the second term on the right-hand side of \eqref{eq:contraction_step_0}, we appeal to Proposition \ref{prop:BluCotGva2023}, as $\eta_{\neq}$ is mean-free.
    More precisely, upon fixing $N$ large enough, there exists an arbitrarily small $\delta>0$ independent of $\kappa$ such that
    \begin{align}
    \begin{split}
        \mathbb{E}_W\biggl|\fint_{\mathbb{T}^2\times\mathbb{T}^2}\mathrm{e}^{\mathrm{i}k\cdot\left(G^n_\kappa(x)-G^n_\kappa(y)\right)}&(P_{\kappa,n}^{(2)})^N[\eta_{\ne}]\begin{pmatrix}
        x+\sqrt{2\kappa}\left(W^2_{n(N+1)+1}-W^2_{n(N+1)}\right)\\
        y+\sqrt{2\kappa}\left(W^2_{n(N+1)+1}-W^2_{n(N+1)}\right)
    \end{pmatrix}\;\mathrm{d}x\,\mathrm{d}y\biggr|\\
        \le&\fint_{\mathbb{T}^2\times\mathbb{T}^2}\left|(P_{\kappa,n}^{(2)})^N[\eta_{\ne}](x,y)\right|\;\mathrm{d}x\,\mathrm{d}y\\
        \le&(2\pi)^{-4}\|V\|_{L^1}\left\|(P_{\kappa,n}^{(2)})^N[\eta_{\ne}]\right\|_V\le\delta\|\eta_{\ne}\|_V,
        \label{eq:contraction_step_2} 
    \end{split}
    \end{align}
    where in the second inequality we have employed Lemma \ref{lemma:prop_of_norm}.
    %Let us remark that in what follows we may re-define $N$ to be even larger, without explicitly mentioning it.
    Putting together the estimates \eqref{eq:contraction_step_0}, \eqref{eq:contraction_step_1} and \eqref{eq:contraction_step_2}, we have reached
    \begin{equation}
        |\overline{\mathcal{T}^{\kappa,n}_k[\eta]}|\le(1-\varepsilon)|\bar{\eta}|+\delta\|\eta_{\ne}\|_V.
        \label{eq:contraction_step_3}
    \end{equation}
    Next, the triangle inequality and the fact that $V\ge1$, together with \eqref{eq:contraction_step_m1}, give
    \begin{align*}
        \|(\mathcal{T}^{\kappa,n}_k[\eta])_{\ne}\|_V=&\sup_{\substack{x,y\in\mathbb{T}^2\\x\ne y}}\frac{1}{V(x,y)}\biggl|\mathbb{E}_W\mathrm{e}^{\mathrm{i}k\cdot\left(G^n_\kappa(x)-G^n_\kappa(y)\right)}\bar{\eta}\\
        &+\mathbb{E}_W\mathrm{e}^{\mathrm{i}k\cdot\left(G^n_\kappa(x)-G^n_\kappa(y)\right)}(P_{\kappa,n}^{(2)})^N[\eta_{\ne}]\begin{pmatrix}
        x+\sqrt{2\kappa}\left(W^2_{n(N+1)+1}-W^2_{n(N+1)}\right)\\
        y+\sqrt{2\kappa}\left(W^2_{n(N+1)+1}-W^2_{n(N+1)}\right)
    \end{pmatrix}\\
        &-\bar{\eta}\fint_{\mathbb{T}^2\times\mathbb{T}^2}\mathbb{E}_W\mathrm{e}^{\mathrm{i}k\cdot\left(G^n_\kappa(x)-G^n_\kappa(y)\right)}\;\mathrm{d}x\,\mathrm{d}y-\fint_{\mathbb{T}^2\times\mathbb{T}^2}\mathbb{E}_W\mathrm{e}^{\mathrm{i}k\cdot\left(G^n_\kappa(x)-G^n_\kappa(y)\right)}(P_{\kappa,n}^{(2)})^N[\eta_{\ne}](x,y)\;\mathrm{d}x\,\mathrm{d}y\biggr|\\
        \le&|\bar{\eta}|+\|(P_{\kappa,n}^{(2)})^N[\eta_{\ne}]\|_V\sup_{\substack{x,y,w\in\T^2\\x\ne y}}\frac{V(x+w,y+w)}{V(x,y)}\\
        &+|\bar{\eta}|+\fint_{\mathbb{T}^2\times\mathbb{T}^2}\left|(P_{\kappa,n}^{(2)})^N[\eta_{\ne}]\begin{pmatrix}
        x+\sqrt{2\kappa}\left(W^2_{n(N+1)+1}-W^2_{n(N+1)}\right)\\
        y+\sqrt{2\kappa}\left(W^2_{n(N+1)+1}-W^2_{n(N+1)}\right)
    \end{pmatrix}\right|\;\mathrm{d}x\,\mathrm{d}y.
    \end{align*}
    The second and the fourth term on the right-hand side can be bounded using Proposition \ref{prop:BluCotGva2023}.
    More precisely, in view of Lemma \ref{lemma:prop_of_norm}, upon fixing a larger $N\in\N$, we have
    \begin{align*}
        \|(P_{\kappa,n}^{(2)})^N[\eta_{\ne}]\|_V\sup_{\substack{x,y,w\in\T^2\\x\ne y}}\frac{V(x+w,y+w)}{V(x,y)}+\fint_{\mathbb{T}^2\times\mathbb{T}^2}|(P_{\kappa,n}^{(2)})^N[\eta_{\ne}](x,y)|\;\mathrm{d}x\,\mathrm{d}y\\
        \le\left(\sup_{\substack{x,y,w\in\T^2\\x\ne y}}\frac{V(x+w,y+w)}{V(x,y)}+(2\pi)^{-4}\|V\|_{L^1}\right)\|(P_{\kappa,n}^{(2)})^N[\eta_{\ne}]\|_V\le\delta\|\eta_{\ne}\|_V,
    \end{align*}
    in light of the fact that $V(x,y)$ has lower and upper bounds that depend on the distance $|x-y|$ only, see \cite[Assumption 2.2]{CooIyeSon2025}.
    Therefore, we have reached the following estimate
    \begin{equation}
        \|(\mathcal{T}^{\kappa,n}_k[\eta])_{\ne}\|_V\le2|\bar{\eta}|+\delta\|\eta_{\ne}\|_V.
        \label{eq:contraction_step_4}
    \end{equation}
    In view of \eqref{eq:contraction_step_3} and \eqref{eq:contraction_step_4}, defining
    \begin{equation*}
        v^n:=\biggl(\underbrace{\begin{bmatrix}
1-\varepsilon & \delta \\
2 & \delta
\end{bmatrix}}_{=:A}\biggr)^n\begin{bmatrix}
|\bar{\eta}|\\
\|\eta_{\ne}\|_V
\end{bmatrix},
    \end{equation*}
     we have that
     \begin{equation*}
         |\overline{(\mathcal{T}^\kappa_k)^{n}[\eta]}|\le v^n_1,\quad\|((\mathcal{T}^\kappa_k)^n[\eta])_{\ne}\|_V\le v_2^n.
     \end{equation*}
    The eigenvalues of the matrix $A$ are
    \begin{equation*}
        \lambda_\pm=\frac{1}{2}\left[1-\varepsilon+\delta\pm\sqrt{(1-\varepsilon)^2+\delta(\delta+2\varepsilon+6)}\right].
    \end{equation*}
    Once $\varepsilon>0$ is given by fixing $g$, one can consequently fix $N$ large enough so that $0<\delta\ll\varepsilon$.
    This choice gives $|\lambda_\pm|<1$, thus $v^n\to\boldsymbol{0}_2$ exponentially fast and the thesis follows for the case $k\ne\boldsymbol{0}_{d-2}$.
    
    Let us now address the case $k=\boldsymbol{0}_{d-2}$, for which we have the additional assumption of $\eta$ being mean-free.
    This leaves us with fewer terms to estimate.
    Upon inspecting the argument above, one can see that in this simpler case it is possible to prove the bound
    \begin{equation*}
        \|\mathcal{T}^{\kappa,n}_k[\eta]\|\le\delta\|\eta\|,
    \end{equation*}
    which concludes the proof.
\end{proof}

\section{Proofs of the main results} \label{sec:proofs}

This section is devoted to the proof of our main result Theorem \ref{thm:main}, whence Corollary \ref{cor:exp_mixer} immediately follows by duality, and to the proof of Corollary \ref{cor:enhanced_dissipation}.

\subsection{Proof of Theorem \ref{thm:main}} \label{subsec:mixing}
Let us fix $q,s,\zeta>0$, with $\zeta$ sufficiently small to be specified later.
% In the following, we will denote by $M>0$ a deterministic constant that may depend on $q,s,\zeta,d$ only, whose value may change from line to line.
In the following, whenever we write $A\lesssim B$ we mean that there exists a deterministic constant $M>0$, possibly depending on $q,s,\zeta,d,N,g$ and the constants $C,\gamma$ in Lemma \ref{lemma:contraction} only, such that $A\le MB$.
Consider the following notation for the Fourier modes
\begin{equation*}
    e_k(x):=\frac{1}{(2\pi)^\frac{d}{2}}\mathrm{e}^{\mathrm{i}k\cdot x},\qquad k\in\mathbb{Z}^d,\quad x\in\T^d.
\end{equation*}
Now, fix $k,k'\in\Z_0^d$.
First, we claim that there exists $\zeta>0$ such that
\begin{equation*}
    N^\kappa_{k,k'}:=\max\left\{n\ge0\,\mid\,\operatorname{Cor}_\kappa^n(e_k,e_{k'})>\mathrm{e}^{-\zeta n}\right\}
\end{equation*}
is almost surely finite.
Note that, if no $n\ge0$ fulfils the condition $\operatorname{Cor}_\kappa^n(e_k,e_{k'})>\mathrm{e}^{-\zeta n}$, then we may simply take $N^\kappa_{k,k'}=0$ by convention.

Indeed, let us compute $\mathbb{P}(N^\kappa_{k,k'}>\ell)$ for any fixed $\ell\in\N$.
By the Chebyshev inequality,
\begin{equation}
    \mathbb{P}\left(N^\kappa_{k,k'}>\ell\right)\le\sum_{j=\ell+1}^\infty\mathbb{P}\left(\operatorname{Cor}_\kappa^j(e_k,e_{k'})>\mathrm{e}^{-\zeta j}\right)\le\sum_{j=\ell+1}^\infty \mathrm{e}^{2\zeta j}\mathbb{E}\left|\int_{\T^d}e_k(x)e_{k'}(T_\kappa^j(x))\;\mathrm{d}x\right|^2.
    \label{eq:correlation_step_0}
\end{equation}
The expected value on the right-hand side of \eqref{eq:correlation_step_0} can be estimated as follows:
\begin{align}
\begin{split}
    \mathbb{E}\left|\int_{\T^d}e_k(x)e_{k'}(T_\kappa^j(x))\;\mathrm{d}x\right|^2=&\int_{\T^d\times\T^d}e_k(x)\overline{e_k(y)}\mathbb{E}\left[e_{k'}(T_\kappa^j(x))\overline{e_{k'}(T_\kappa^j(y))}\right]\;\mathrm{d}x\,\mathrm{d}y\\
    \le&(2\pi)^{-d}\int_{\T^d\times\T^d}\left|\mathbb{E}\left[e_{k'}(T_\kappa^j(x))\overline{e_{k'}(T_\kappa^j(y))}\right]\right|\;\mathrm{d}x\,\mathrm{d}y\\
    =&(2\pi)^{-2}\int_{\T^2\times\T^2}\left|(\mathcal{T}^{\kappa}_{(k'_3,\dots,k'_d)})^j[e_{(k'_1,k'_2)}^{(2)}](x,y)\right|\;\mathrm{d}x\,\mathrm{d}y,
    \label{eq:correlation_step_1}
\end{split}
\end{align}
Now, taking into account \eqref{eq:correlation_step_1}, appealing to Lemmas \ref{lemma:prop_of_norm} and \ref{lemma:contraction}, we get the existence of constants $\gamma,c>0$ independent of $\kappa$ such that
\begin{equation}
    \mathbb{E}\left|\int_{\T^d}e_k(x)e_{k'}(T_\kappa^j(x))\;\mathrm{d}x\right|^2\le(2\pi)^{-2}\|V\|_{L^1}\left\|(\mathcal{T}^{\kappa}_{(k'_3,\dots,k'_d)})^j[e_{(k'_1,k'_2)}^{(2)}]\right\|\le c\mathrm{e}^{-\gamma j}\|e_{(k'_1,k'_2)}^{(2)}\|.
    \label{eq:correlation_step_2}
\end{equation}
Let us be more precise about the last inequality.
If $(k'_1,k'_2)\ne\boldsymbol{0}_2$, the application of Lemma \ref{lemma:contraction} is straightforward, as $e_{(k'_1,k'_2)}^{(2)}$ is mean-free.
Now, let us consider the case of $k'_1=k'_2=0$.
Since $k'\in\Z^d_0$, there exists $q\in\{3,\dots,d\}$ such that $k'_q\ne0$.
This implies that $(k'_3,\dots,k'_d)\ne\boldsymbol{0}_{d-2}$, thus Lemma \ref{lemma:contraction} applies also in this case.

Putting \eqref{eq:correlation_step_0} and \eqref{eq:correlation_step_2} together, and further observing that $\|e_{(k'_1,k'_2)}^{(2)}\|\le3(2\pi)^{-2}$, we infer
\begin{equation}
    \mathbb{P}\left(N^\kappa_{k,k'}>\ell\right)\le3(2\pi)^{-2}c\sum_{j=\ell+1}^\infty\mathrm{e}^{(2\zeta-\gamma)j}\lesssim\mathrm{e}^{(2\zeta-\gamma)\ell}.
    \label{eq:P_N_ge_ell}
\end{equation}
Picking any $\zeta<\gamma/2$ ensures that $N^\kappa_{k,k'}$ is almost surely finite.
In view of \eqref{eq:P_N_ge_ell}, we may employ a standard argument to conclude the proof of Theorem \ref{thm:main}, see for instance the recent works \cite[Proposition 4.1]{NAVARROFERNANDEZ2026111227} or \cite[Lemma 3.3]{CooIyeSon2025}.
For the sake of completeness, we report below the key points of such argument.

Let
\begin{equation*}
    K_\kappa:=\left\lceil\max\left\{\max\{|k|,|k'|\}\,\mid\,\mathrm{e}^{\zeta N^\kappa_{k,k'}}>|k||k'|\right\}\right\rceil,
\end{equation*}
with the convention that $K_\kappa=0$ if no $k,k'\in\Z^d_0$ fulfil the condition $\mathrm{e}^{\zeta N^\kappa_{k,k'}}>|k||k'|$.
Then, use \eqref{eq:P_N_ge_ell} to compute the following:
\begin{align}
\begin{split}
    \mathbb{P}(K_{\kappa}>\ell)\le&\sum_{|k|\ge1,|k'|>\ell}\mathbb{P}(\mathrm{e}^{\zeta N^\kappa_{k,k'}}>|k||k'|)+\sum_{|k|>\ell,|k'|\ge1}\mathbb{P}(\mathrm{e}^{\zeta N^\kappa_{k,k'}}>|k||k'|)\\
    \lesssim&\sum_{|k|\ge1}|k|^{2-\frac\gamma\zeta}\sum_{|k'|>\ell}|k'|^{2-\frac\gamma\zeta}\lesssim\ell^{2+d-\frac\gamma\zeta},
    \label{eq:P_K_ge_ell}
\end{split}
\end{align}
where the last inequality holds true provided that we pick
\begin{equation*}
    \zeta<\frac{\gamma}{d+2}<\frac{\gamma}{2}.
\end{equation*}
This condition also implies that $\lim_{\ell\to\infty}\mathbb{P}(K_{\kappa}>\ell)=0$, therefore $K_\kappa$ is almost surely finite.
%\textcolor{red}{[is this enough or we have to further restrict $\zeta$ in order to apply again Borel--Cantelli?]}
% Consider the events
% \begin{equation*}
%     E^\kappa_{k,k'}=\left\{\mathrm{e}^{\zeta N^\kappa_{k,k'}}>|k||k'|\right\}.
% \end{equation*}
% We have
% \begin{equation*}
%     \mathbb{P}(E^\kappa_{k,k'})=\mathbb{P}\left(N^\kappa_{k,k'}>\frac1\zeta\left(\log|k|+\log|k'|\right)\right)\le3c(|k||k'|)^{2-\frac\gamma\zeta},
% \end{equation*}
% whence
% \begin{equation*}
%     \sum_{k,k'}\mathbb{P}(E^\kappa_{k,k'})\le3c\left(\sum_k|k|^{2-\frac\gamma\zeta}\right)^2<\infty,
% \end{equation*}
As a consequence, defining
\begin{equation*}
    C_\kappa:=\max_{|k|,|k'|\le K_\kappa}\mathrm{e}^{\zeta N_{k,k'}^\kappa},
\end{equation*}
we have the almost-sure estimate
\begin{equation*}
    \mathrm{e}^{\zeta N^\kappa_{k,k'}}\le C_\kappa|k||k'|.
\end{equation*}
Hence, denoting by $\varphi_k,\psi_{k'}$ the Fourier coefficients of $\varphi,\psi$ respectively, we get
\begin{align*}
    \left|\int_{\T^d}\varphi(x)\psi(T_\kappa^n(x))\;\mathrm{d}x\right|\le&\sum_{k,k'}|\varphi_k||\psi_{k'}|\operatorname{Cor}^n_\kappa(e_k,e_{k'})\le\sum_{k,k'}|\varphi_k||\psi_{k'}|\mathrm{e}^{\zeta(N^\kappa_{k,k'}-n)}\\
    \le&C_\kappa\mathrm{e}^{-\zeta n}\sum_{k}|k||\varphi_k|\sum_{k'}|k'||\psi_{k'}|\lesssim C_{\kappa}\mathrm{e}^{-\zeta n}\|\varphi\|_{H^s}\|\psi\|_{H^s},
\end{align*}
where the last inequality holds for all $s>1+d/2$, so in particular
\begin{equation*}
    \left|\int_{\T^d}\varphi(x)\psi(T_\kappa^n(x))\;\mathrm{d}x\right|\lesssim C_{\kappa}\mathrm{e}^{-\zeta n}\|\varphi\|_{H^{2+\frac d2}}\|\psi\|_{H^{2+\frac d2}}.
\end{equation*}
Therefore, in view of \cite[Lemma 7.1]{BedBluPun2022} and the fact that $C_\kappa\ge1$ almost surely, for all $\varepsilon>0$ and any $s\in(0,1+d/2)$ we have
\begin{equation}
    \left|\int_{\T^d}\varphi(x)\psi(T_\kappa^n(x))\;\mathrm{d}x\right|\lesssim C_{\kappa}\left(\mathrm{e}^{-\zeta n}\varepsilon^{-2\frac{2-s+d/2}{s}}+2\varepsilon\right)\|\varphi\|_{H^s}\|\psi\|_{H^s}.
    \label{eq:hidden_constant_M}
\end{equation}
%\textcolor{red}{[This is maybe good to write since there's probably a typo in \cite{BedBluPun2022}] 2 typos\\}
Now, choosing $\varepsilon=\mathrm{e}^{-\bar{\gamma}n}$ with $\bar{\gamma}:=s\zeta/(d+4)$, we obtain the almost-sure estimate
\begin{equation*}
    \left|\int_{\T^d}\varphi(x)\psi(T_\kappa^n(x))\;\mathrm{d}x\right|\le D_\kappa\mathrm{e}^{-\bar{\gamma} n}\|\varphi\|_{H^s}\|\psi\|_{H^s},
\end{equation*}
where $D_\kappa:=3MC_\kappa$, with $M$ being the hidden constant in \eqref{eq:hidden_constant_M}.
There is only left to prove that $C_\kappa$ has finite moments, independently of $\kappa$.
To this regard, we use the Cauchy--Schwarz inequality, \eqref{eq:P_N_ge_ell}, \eqref{eq:P_K_ge_ell}, and the fact that the elements $k\in\Z^d$ with $|k|\le\ell$ are $\mathcal{O}(\ell^d)$ to infer
\begin{align*}
    \mathbb{E}|C_\kappa|^q&=\sum_{\ell=1}^\infty\mathbb{E}\left[\mathbbm{1}_{K_\kappa=\ell}\max_{|k|,|k'|\le\ell}\mathrm{e}^{q\zeta N_{k,k'}^\kappa}\right]\\
    &\le\sum_{\ell=1}^\infty\sqrt{\mathbb{P}(K_\kappa>\ell-1)}\sqrt{\sum_{|k|,|k'|\le\ell}\left(1+\sum_{j=1}^\infty\mathbb{P}\left(N_{k,k'}^\kappa>\frac{\log{j}}{2q\zeta}\right)\right)}\\
    &\lesssim1+\sum_{\ell=1}^\infty\ell^{1+\frac32d-\frac{\gamma}{2\zeta}}<\infty,
\end{align*}
% where the last inequality holds true provided that
% \begin{equation*}
%     \zeta<\min\left\{\frac{\gamma}{2(1+qd)},\frac{\gamma}{d+2}\right\}.
% \end{equation*}
% Finally, further restricting to
provided that
\begin{equation*}
    \zeta<\min\left\{\frac{\gamma}{2(1+qd)},\frac{\gamma}{2+5d}\right\}.
\end{equation*}
The proof of Theorem \ref{thm:main} is then concluded.

\subsection{Proof of Corollary \ref{cor:enhanced_dissipation}} \label{subsec:enh_diss}

The proof of Corollary \ref{cor:enhanced_dissipation} follows closely the proof of \cite[Theorem 1.1]{CooIyeSon2025}.
Let us fix any $\varepsilon>0$.
In the following, whenever we write $A\lesssim B$ we mean that there exists a deterministic constant $M>0$, possibly depending on $\varepsilon,d,N,g$ and the constants $C,\gamma$ in Lemma \ref{lemma:contraction} only, such that $A\le MB$.
Recall that $N$ is the number of time steps appearing in the definition of $u_{\underline{\omega}}$ in \eqref{eq:def_of_u}.
First, arguing as in the proof of \cite[Lemma 5.6]{ConKisRyzZla2008} (cf. Equation $(5.8)$ therein and below), we have
\begin{equation}
    \|\rho(n(N+1)+N+1)\|_{L^\infty}\lesssim\kappa^{-\frac{d}{4}-\frac\varepsilon4}\|\rho(n(N+1)+N)\|_{L^2}.
    \label{eq:enh_diss_step_m2}
\end{equation}
We now use \cite[Lemma 4.18]{BluPun2023} to infer
\begin{equation}
    \|\rho(n(N+1)+N)\|_{L^2}\lesssim\kappa^{-\frac\varepsilon4}\|\rho(n(N+1))\|_{H^{-\frac{\varepsilon}{2}}}.
    \label{eq:enh_diss_step_m1}
\end{equation}
Putting together \eqref{eq:enh_diss_step_m2} and \eqref{eq:enh_diss_step_m1}, one gets
\begin{equation}
    \|\rho(n(N+1)+N+1)\|_{L^\infty}\lesssim\kappa^{-\frac{d+2\varepsilon}{4}}\|\rho(n(N+1))\|_{H^{-\frac\varepsilon2}}.
    \label{eq:enh_diss_step_0}
\end{equation}
By Theorem \ref{thm:main}, there exist a constant $\gamma>0$ and a positive random variable $\tilde D_\kappa$ over $(\Omega,\mathcal{F},\mathbb{P})$ such that $\mathbb{E}|\tilde D_\kappa|^q<\infty$ uniformly in $\kappa$ and
\begin{align}
\begin{split}
    \|\rho(n(N+1))\|_{H^{-\frac\varepsilon2}}\le&\mathbb{E}_W\|\rho(N+1)\circ(Z_\kappa^1)^{-1}\circ\cdots\circ(Z_\kappa^{n-1})^{-1}\|_{H^{-\frac\varepsilon2}}\\
    \le&(\underbrace{\mathbb{E}_W\tilde D_\kappa}_{=:D_\kappa})\mathrm{e}^{-\gamma(n-1)}\|\rho(N+1)\|_{H^{\frac\varepsilon2}}.
    \label{eq:enh_diss_step_1}
\end{split}
\end{align}
One can check that, for $q\ge1$, the Jensen inequality implies
\begin{equation*}
    \mathbb{E}_0|D_\kappa|^q=\mathbb{E}_0|\mathbb{E}_W\tilde D_\kappa|^q\le\mathbb{E}|\tilde D_\kappa|^q<\infty,
\end{equation*}
uniformly in $\kappa$.
On the other hand, for $q\in(0,1)$ the Jensen inequality again gives
\begin{equation*}
    \mathbb{E}_0|D_\kappa|^q=\mathbb{E}_0|\mathbb{E}_W\tilde D_\kappa|^q\le|\mathbb{E}\tilde D_\kappa|^q<\infty,
\end{equation*}
uniformly in $\kappa$.
Now, \cite[Lemma 4.17]{BluPun2023} implies
\begin{equation}
    \|\rho(N+1)\|_{H^{\frac\varepsilon2}}\lesssim\kappa^{-\frac \varepsilon4}\|\rho(1)\|_{L^2}\lesssim\kappa^{-\frac\varepsilon4}\kappa^{-\frac{d}{4}-\frac\varepsilon4}\|\rho_0\|_{L^1},
    \label{eq:enh_diss_step_2}
\end{equation}
where in the last inequality we have argued again as in the proof of \cite[Lemma 5.6]{ConKisRyzZla2008}.
Combining the estimates \eqref{eq:enh_diss_step_0}, \eqref{eq:enh_diss_step_1} and \eqref{eq:enh_diss_step_2} we obtain
\begin{equation*}
    \|\rho(n(N+1)+N+1)\|_{L^\infty}\lesssim\frac{D_\kappa}{\kappa^{\frac{d}{2}+\varepsilon}}\mathrm{e}^{-\gamma(n-1)}\|\rho_0\|_{L^1},
\end{equation*}
whence the maximum principle gives
\begin{equation*}
    \|\rho(t)\|_{L^\infty}\lesssim\frac{D_\kappa}{\kappa^{\frac{d}{2}+\varepsilon}}\mathrm{e}^{-\gamma(\frac{t}{N+1}-2)}\|\rho_0\|_{L^1},\quad t\ge N+1.
\end{equation*}
Up to rescaling $D_\kappa$, this concludes the proof of Corollary \ref{cor:enhanced_dissipation}.

\section*{Acknowledgements}
The author is deeply grateful to David Villringer for suggesting the problem and for his constant support throughout the development of this work.
The author also thanks Michele Coti Zelati for valuable comments on an earlier version of the manuscript, and Víctor Navarro-Fernández for useful discussions.
The author acknowledges the use of AI tools to refine the assumption in Lemma \ref{lemma:unif_bdd_from_1} and for proofreading.
The research of LM was supported by EPSRC under grant number EP/Y03533X/1, with additional support from the Department of Mathematics at Imperial College London.

% \section*{Declaration of AI use}

\bibliographystyle{abbrvnat}
\bibliography{ref}
\end{document}